\documentclass[a4paper,11pt]{amsart}

\usepackage{amssymb}
\usepackage{amstext}
\usepackage{amsmath}
\usepackage{amscd}
\usepackage{amsthm}
\usepackage{amsfonts}
\usepackage{enumerate}
\usepackage{graphicx}
\usepackage{color}
\usepackage{here} % 図や表を強制的に出力させる
\usepackage{bm} % 太字ベクトル
\usepackage{mathrsfs}
\usepackage{mathtools}
\usepackage{latexsym}
\usepackage{booktabs}
\usepackage{fancyhdr}
\usepackage{subcaption}
\usepackage{enumitem}
\usepackage{tikz}
\usepackage{tikz-cd}
\usetikzlibrary{arrows}
\usetikzlibrary{decorations.markings}
\usetikzlibrary{patterns,decorations.pathreplacing}
\usepackage{mathtools}
\usepackage{pdfpages}
\usepackage[section]{algorithm}
\usepackage{algpseudocodex}

\usepackage{hyperref}
\hypersetup{
  colorlinks=true,
  citecolor=red,
  linkcolor=blue
}
\theoremstyle{plain}
\newtheorem{theorem}{Theorem}[section]
\newtheorem{lemma}[theorem]{Lemma}
\newtheorem{proposition}[theorem]{Proposition}
\newtheorem{corollary}[theorem]{Corollary}

\theoremstyle{definition}

\theoremstyle{remark}

\newcommand{\kk}{\Bbbk}

\newcommand{\lessdotrel}{\mathrel{\lessdot}}

\newcommand{\JJ}{\operatorname{J}}
\newcommand{\MM}{\operatorname{M}}

\DeclareMathOperator{\ini}{in}

\title{Algebras with straightening laws on join- or meet-semidistributive lattices}

\author{Koji Matsushita}
\address{Graduate School of Mathematical Sciences,
The University of Tokyo,
Komaba, Meguro-ku, Tokyo 153-8914, Japan}
\email{koji-matsushita@g.ecc.u-tokyo.ac.jp}

\author{Sora Miyashita}
\address{Department of Pure and Applied Mathematics,
Graduate School of Information Science and Technology,
The University of Osaka,
Suita, Osaka 565-0871, Japan}
\email{u804642k@ecs.osaka-u.ac.jp}

\author{Koichiro Tani}
\address{Department of Pure and Applied Mathematics,
Graduate School of Information Science and Technology,
The University of Osaka,
Suita, Osaka 565-0871, Japan}
\email{tani-k@ist.osaka-u.ac.jp}

\begin{document}

\begin{abstract}
We study algebras with straightening laws on join- or meet-semidistributive lattices.
We show that for a join-semidistributive (resp. meet-semidistributive) lattice, meet-distributivity (resp. join-distributivity) and Cohen--Macaulayness are equivalent, and that integrality implies these conditions. Thus, Hibi's conjecture that every integral lattice is Cohen--Macaulay holds for join- or meet-semidistributive lattices. For semidistributive lattices, distributivity, integrality and Cohen--Macaulayness are equivalent.
\end{abstract}

\maketitle

\section{Introduction}

We assume that all posets and lattices appearing in this paper are finite.

Algebras with straightening laws (ASLs) provide a useful framework for studying
commutative algebras whose structure is governed by a partially ordered set.
A poset $P$ is called \emph{integral} if there exists a homogeneous ASL domain
on $P$ over a field.
It is therefore natural to ask which posets, and in particular which lattices,
are integral.
Hibi proved that every distributive lattice is integral and conjectured that
every integral lattice is Cohen--Macaulay~\cite{H87}.
Although the analogous statement for arbitrary posets is false
~\cite{Terai94}, the conjecture for lattices provides a natural problem
relating the algebraic properties of ASLs to the combinatorial structure of
lattices.

In this paper, we focus on join- or meet-semidistributive lattices.
These form fundamental classes in lattice theory, and semidistributive lattices
also arise naturally in representation theory; for example, the lattice of
torsion classes of a finite-dimensional algebra is completely
semidistributive~\cite{DIRRT}.

Our main result describes Cohen--Macaulayness and integrality in terms of
meet- or join-distributivity:
\begin{theorem}\label{thm:main1}
    Let $L$ be a join-semidistributive (resp. meet-semidistributive) lattice.
    Then the following statements hold:
\begin{enumerate}[label=(\arabic*),font=\normalfont,leftmargin=*]
        \item[(1)] $L$ is meet-distributive (resp. join-distributive) if and only if $L$ is Cohen--Macaulay over an arbitrary field. 
        % \begin{itemize}
        %     \item[(i)] $L$ is meet-distributive (resp. join-distributive);
        %     \item[(ii)] $L$ is Cohen--Macaulay over an arbitrary field;
        %     \item[(iii)] every ASL on $L$ satisfies Serre's condition $(S_2)$ over every field;
        % \end{itemize}
        \item[(2)] If $L$ is integral, then $L$ is meet-distributive (resp. join-distributive).
        \item[(3)] There exists a join-semidistributive lattice $L$ that is meet-distributive but not integral.
        \item[(4)] There exists a join-semidistributive lattice $L$ that is integral but not distributive.
    \end{enumerate}
\end{theorem}

Theorem~\ref{thm:main1}~(1) and (2) give the following partial affirmative
answer to Hibi's conjecture.
In fact, the conclusion is stronger in the sense that Cohen--Macaulayness
holds over every field.

\begin{corollary}
    Let $L$ be a join- or meet-semidistributive lattice.
    If $L$ is integral, then $L$ is Cohen--Macaulay over an arbitrary field.
\end{corollary}

We also obtain a particularly simple characterization in the semidistributive
case.
Indeed, a lattice that is both meet- and join-distributive is
semidistributive and modular, and hence distributive.
Moreover, every distributive lattice is integral~\cite{H87}.
Together with Theorem~\ref{thm:main1}, this yields the following:

\begin{corollary}\label{cor:main2}
Let $L$ be a semidistributive lattice.
The following conditions are equivalent:
\begin{enumerate}[label=(\arabic*),font=\normalfont,leftmargin=*]
\item[(i)] $L$ is distributive;
\item[(ii)] $L$ is Cohen--Macaulay over an arbitrary field;
\item[(iii)] $L$ is integral.
% \item[(i)] every ASL on $L$ is Cohen--Macaulay over every field;
% \item[(ii)] every ASL on $L$ satisfies Serre's condition $(S_2)$ over every field;
% \item[(iii)] $L$ is integral;
% \item[(iv)] $L$ is distributive;
\end{enumerate}
\end{corollary}

%The main lattice-theoretic ingredient in the proof is a characterization of meet-distributivity in terms of the maximal-chain exchange graph.
The main lattice-theoretic ingredient in the proof is a characterization, in terms of the maximal-chain exchange graph, of meet-distributivity.
More precisely, we prove that a join-semidistributive lattice is
meet-distributive if and only if its maximal-chain exchange graph is connected (Theorem~\ref{thm:char_MD}).
The dual statement holds for meet-semidistributive lattices.
This characterization, together with known connectivity properties of
Cohen--Macaulay and integral posets, yields Theorem~\ref{thm:main1}~(1) and (2).

\medskip

This paper is organized as follows.
In Section~\ref{sec:ASL}, we recall basic facts on ASLs and order complexes, including their Cohen--Macaulayness and $h$-vectors.
In Section~\ref{sec:lattice}, we study the maximal-chain exchange graph of a semidistributive lattice and prove the characterization of meet- and join-distributivity described above.
In Section~\ref{sec:proof}, we prove Theorem~\ref{thm:main1}.

\section{ASLs and order complexes}\label{sec:ASL}
In this section, we recall basic facts on ASLs and order complexes, together
with some results on Cohen--Macaulayness, $h$-vectors, and maximal-chain
exchange graphs that will be used later.
For basic facts on algebras with straightening laws and Stanley--Reisner
theory, we refer the reader to~\cite{EisenbudASL,StanleyCCA}.

Let $R = \bigoplus_{n\ge0} R_n$ be a Noetherian graded algebra over a field $R_0=\kk$.
Let $P$ be a poset and suppose that an injection $\varphi: P \hookrightarrow \bigoplus_{n\ge 1} R_n$ for which the $\kk$-algebra $R$ is generated by $\varphi(P)$ over $\kk$ is given.
%A {\em standard monomial} is a homogeneous element of $R$ of the form $\varphi(\gamma_1) \varphi(\gamma_2)\cdots \varphi(\gamma_s)$, where $\gamma_1 \leq \gamma_2 \leq \cdots \leq \gamma_s$ in $P$. 
A \emph{standard monomial} is an element of $R$ of the form $\varphi(\gamma_1)\varphi(\gamma_2)\cdots\varphi(\gamma_s)$,
where $s\geq 1$ and $\gamma_1\leq\gamma_2\leq\cdots\leq\gamma_s$ in $P$.
The empty product $1$ is also regarded as a standard monomial.
We call $R$ an \emph{algebra with straightening laws}($=$\emph{ASL}) on $P$ over $\kk$ if the following conditions are satisfied:
\begin{enumerate}[label=(A\arabic*), ref=A\arabic*]
\item\label{A1}
The set of standard monomials is a $\kk$-basis of $R$;
\item\label{A2}
If $\alpha$ and $\beta$ in $P$ are incomparable and if
\begin{align*}\label{ASL}
\varphi(\alpha)\varphi(\beta)
= \sum_{i} r_i\,\varphi(\gamma_{i_1})\varphi(\gamma_{i_2}) \cdots \varphi(\gamma_{i_p}) , \quad 0 \neq r_i \in \kk, \quad \gamma_{i_1}\leq \gamma_{i_2} \leq \cdots \leq \gamma_{i_p},
\end{align*}
is the unique expression for $\varphi(\alpha)\varphi(\beta) \in R$ as a linear combination of distinct standard monomials guaranteed by (\ref{A1}), then $\gamma_{i_1} \leq \alpha, \beta$ for every $i$.
\end{enumerate}
Note that the right-hand side of the relation in (\ref{A2}) is allowed to be the empty sum $(=0)$.
%We abbreviate an algebra with straightening laws as ASL. The relations in (A2) are called the {\em straightening relations} for $R$.

We say that a Noetherian graded algebra $R = \bigoplus_{n\ge 0} R_n$ over a field $R_0=\kk$ is \emph{homogeneous} if $R=\kk[R_1]$.
A poset $P$ is called {\em integral} if there exists an ASL domain
$R=\bigoplus_{n\ge0}R_n$ on $P$ over a field $R_0=\kk$ with an injection
$\varphi:P\hookrightarrow R_1$.
In particular, $R$ is a homogeneous domain.

Suppose that $R = \bigoplus_{n\ge 0} R_n$ is a homogeneous algebra over a field $R_0=\kk$.
The \textit{Hilbert series} of $R$ is defined as the formal power series $\sum_{n\ge 0}(\dim_\kk R_n)t^n$, and it is known that we can write the Hilbert series of $R$ as the following form:
\[
    \sum_{n \ge 0} (\dim_\kk R_n) t^n=\frac{h_0+h_1t+\cdots+h_st^s}{(1-t)^{\dim R}},
\]
where $h_s \neq 0$.
We call the sequence $(h_0,h_1,\ldots,h_s)$ the \textit{$h$-vector} of $R$, denote it by $h(R)$.
We refer the reader to \cite{S78} for detailed information on $h$-vector of homogeneous algebras.

% Let $R = \bigoplus_{n=0}^{\infty} R_n$ be a homogeneous ASL on a poset $P$ over a field $R_0=K$ with an injection $\varphi: P \hookrightarrow R_1$ and $\Delta(P)$ the order complex of $P$.  It follows from (ASL-1) that the $h$-vector of $R$ is equal to the nonzero components of the $h$-vector $h(\Delta(P))$ of $\Delta(P)$.  In other words, if $h(\Delta(P))=(h_0,h_1, \ldots, h_s, 0,\ldots,0)$ with $h_s \neq 0$, then $h(R) = (h_0,h_1,\ldots, h_s)$.  We refer the reader to \cite{StaGREEN} and \cite{HIBIred} for the background on combinatorics of simplicial complexes and their $h$-vectors.

\medskip

A \emph{chain} of $P$ is a subset of pairwise comparable elements, and a \emph{maximal chain} is a chain maximal under inclusion.
The \emph{order complex} $\Delta(P)$ is the simplicial complex whose faces are the chains of $P$ ; its \emph{facets}, that is, its maximal faces, are the maximal chains. A simplicial complex is \emph{pure} if all its facets have the same dimension.
For a pure simplicial complex, the \emph{facet-adjacency graph} has the facets as vertices, with two facets adjacent when their intersection has codimension one in each.
A pure simplicial complex is \emph{shellable} if its facets can be ordered $F_1,\ldots,F_m$ so that, for every $j>1$, the intersection $F_j\cap\bigcup_{i<j}F_i$ is pure of codimension one in $F_j$; in particular, it is \emph{strongly connected}, that is, its facet-adjacency graph is connected.
The \emph{maximal-chain exchange graph} $G(P)$ has the maximal chains as vertices. Two maximal chains $C,C'$ are adjacent when \(|C\setminus C'|=|C'\setminus C|=1\).
When $\Delta(P)$ is pure, this is precisely its facet-adjacency graph.

\medskip

Let $S_P=\kk[x_{\alpha} : \alpha \in P]$ denote the polynomial ring in $|P|$ variables over $\kk$.
We denote the Stanley--Reisner ideal and Stanley--Reisner ring of $\Delta(P)$ by
\[
J_P:=(x_\alpha x_\beta : \alpha, \beta\in P \text{ are incomparable}) \quad  \text{ and }\quad \kk[\Delta(P)]:=S_P/J_P.
\]
Note that $\kk[\Delta(P)]$ is a homogeneous ASL on $P$ over $\kk$, called the \emph{discrete ASL}.
We say that $P$ is \emph{Cohen--Macaulay over $\kk$} if $\kk[\Delta(P)]$ is Cohen--Macaulay.

Suppose that $R=\bigoplus_{n\ge0}R_n$ is an ASL on $P$ over a field
$R_0=\kk$ with an injection
$\varphi:P\hookrightarrow\bigoplus_{n\ge1}R_n$.
Define the surjective ring homomorphism $\pi:S_P\to R$ by
$\pi(x_\alpha)=\varphi(\alpha)$, and set $I_R:=\mathrm{Ker}(\pi)$, so that
$R\cong S_P/I_R$.
Give $S_P$ the grading $\deg x_\alpha=\deg\varphi(\alpha)$.
Then $I_R$ is homogeneous with respect to this grading, and, with respect
to a (weighted) degree reverse lexicographic order associated to a total
order on $P$ refining its partial order, one has $\operatorname{in}(I_R)=J_P$ 
(see, e.g., \cite[Section~3.1]{CV20} or \cite[Section~5]{C07}).
%(\cite[Remark~3.8]{CV20}, citing \cite[Lemma~5.5]{C07}).
Thus, the discrete ASL $\kk[\Delta(P)]$ is a Gr\"obner degeneration of $R$.
This leads to the following important facts:

\begin{proposition}\label{prop:h-vector}
    Let $P$ be a poset and let $R=\bigoplus_{n\ge 0}R_n$ be a homogeneous ASL on $P$ over a field $R_0=\kk$.
    Then one has $h(R)=h(\kk[\Delta(P)])$.
\end{proposition}

\begin{theorem}[{\cite[Corollary~3.9]{CV20}}]\label{thm:universal-asl}
Let $P$ be a poset and let $R=\bigoplus_{n\ge 0}R_n$ be an ASL on $P$ over a field $R_0=\kk$.
Then $R$ is Cohen--Macaulay if and only if $P$ is Cohen--Macaulay.
\end{theorem}

\begin{theorem}[{\cite[Proposition~11.7]{B95} and \cite[Theorem~1]{KS95}}]\label{thm:universal}
Let $P$ be a poset.
If either of the following two conditions are satisfied:
\begin{enumerate}[label=(\arabic*),font=\normalfont,leftmargin=*]
    \item[(1)] $P$ is Cohen--Macaulay over some field,
    \item[(2)] $P$ is integral,
\end{enumerate}
then $\Delta(P)$ is pure and strongly connected, equivalently, $G(P)$ is connected.
\end{theorem}

% \begin{proof}
% Write $A=S_P/\mathfrak p$. The ideal $\mathfrak p$ is prime, and \eqref{eq:asl-degeneration} gives $\ini_{\prec}(\mathfrak p)=J_P$. Since $J_P$ is squarefree, the initial complex is exactly $\Delta(P)$. Kalkbrener and Sturmfels~\cite[Theorem~1]{KS95} prove that the initial complex of a prime ideal is pure and connected through codimension-one faces. Hence $G(P)$ is connected.
% \end{proof}

\bigskip

\section{Semidistributivity and maximal-chain exchange graphs}\label{sec:lattice}
%In this section, we study maximal-chain exchange graphs of join- and meet-semidistributive lattices.
In this section, we study maximal-chain exchange graphs for join-semidistributive lattices and for meet-distributive lattices.
Our main goal is to characterize meet- and join-distributivity in terms of
the connectivity of these graphs.

\subsection{Lattice-theoretic preliminaries}
We first recall the lattice-theoretic notions used below.
For basic terminology and results on lattice theory, we refer the reader
to~\cite{AdarichevaNationSemidistributive,EdelmanMeetDistributive,GraetzerLatticeTheory}.

Let $L$ be a lattice with the minimum element $\widehat0$ and the maximum element $\widehat1$. The \emph{dual lattice} $L^d$ is obtained by reversing the order. For $x\le y$, the \emph{interval} $[x,y]$ is the subposet $\{z\in L:x\le z\le y\}$. If $x<y$ and no element lies strictly between them, we write $x\lessdotrel y$ and call this a \emph{cover relation}. A chain $x_0<\cdots<x_m$ has \emph{length} $m$ and is \emph{saturated} if $x_{i-1}\lessdotrel x_i$ for every $i$.
%The length $\ell(P)$ of a finite poset is the maximum length of a chain in $P$. A finite lattice is \emph{graded} if every interval has maximal chains of one common length; for a graded lattice, $\operatorname{rank}L$ denotes the common length of its maximal chains. The \emph{proper part} of $L$ is $\overline L=L\setminus\{\widehat0,\widehat1\}$.
An element $j\ne\widehat0$ is \emph{join-irreducible} if $j=a\vee b$ implies $j=a$ or $j=b$. In a finite lattice this is equivalent to $j$ having a unique lower cover, denoted $j_*$. Let $\JJ(L)$ be the set of join-irreducibles, and let $\MM(L)$ be the set of \emph{meet-irreducibles}, defined dually.

We recall the lattice classes that will be used throughout the remainder of the paper.
\begin{itemize}
    \item The lattice $L$ is \emph{distributive} if for all $a,b,c\in L$, one has
\[
a\wedge(b\vee c)=(a\wedge b)\vee(a\wedge c) \quad \text{ and }\quad a\vee(b\wedge c)=(a\vee b)\wedge(a\vee c).
\]

   \item The lattice $L$ is \emph{join-semidistributive} (\emph{JSD}, for short) if for all $a,b,c\in L$,
\[
a\vee b=a\vee c
\quad\Longrightarrow\quad
a\vee(b\wedge c)=a\vee b,
\]
and \emph{meet-semidistributive} (\emph{MSD}, for short) if the dual condition holds.
It is \emph{semidistributive} (\emph{SD}, for short) if it is JSD and MSD.

   \item The lattice $L$ is \emph{meet-distributive} (\emph{MD}, for short) if $[x_\downarrow,x]$ is a Boolean lattice for every $x\in L\setminus\{\hat{0}\}$, where $x_\downarrow:=\bigwedge\{y\in L:y\lessdot x\}$, and \emph{join-distributive} (\emph{JD}, for short) if the dual condition holds.
   
   \item The lattice $L$ is \emph{lower semimodular} if, for all $x,y\in L$,
\[
x\lessdotrel x\vee y
\quad\Longrightarrow\quad
x\wedge y\lessdotrel y.
\]
and is \emph{upper semimodular} if $L^d$ is lower semimodular.
% \[
% x\wedge y\lessdotrel x
% \quad\Longrightarrow\quad
% y\lessdotrel x\vee y
% \qquad(x,y\in L).
% \]
\end{itemize}

\begin{proposition}[{\cite[Proposition~2.1]{Cz14}}]\label{prop:cz}
   Let $L$ be a lattice. Then the following conditions are equivalent:
\begin{enumerate}[label=(\arabic*),font=\normalfont,leftmargin=*]
       \item[(i)] $L$ is MD (resp. JD);
       \item[(ii)] $L$ is JSD (resp. MSD) and lower semimodular (resp. upper semimodular);
       \item[(iii)] every maximal chain has the length $|\JJ(L)|$ (resp. $|\MM(L)|$).
   \end{enumerate}
\end{proposition}

\subsection{A characterization of meet-distributivity}

Throughout this subsection, we assume that $L$ is JSD.
We first recall the standard edge labeling of a JSD lattice.
For a cover relation $x\lessdotrel y$, there exists the unique minimal element of $\{z\in L:x\vee z=y\}$ since $L$ is JSD.
We denote this element by $\lambda(x,y)$.
It is well known that $\lambda(x,y)\in\JJ(L)$; see, for example,
\cite[Lemma~3.1]{MuehleMeetDistributive}.
Moreover, $\lambda(j_*,j)=j$ for every $j\in\JJ(L)$.
We refer to $\lambda$ as the \emph{canonical join-irreducible labeling} of $L$.

\begin{lemma}\label{lem:diamond-label}
For $u,p,q,v\in L$ with $u\lessdotrel p,q\lessdotrel v$ and $p\neq q$, we have $\lambda(u,p)=\lambda(q,v)$ and $\lambda(u,q)=\lambda(p,v)$.
\end{lemma}

\begin{proof}
Set \(j=\lambda(u,p)\) and \(j'=\lambda(q,v)\).
Since \(u\vee j=p\), we have \(q\vee j=q\vee p=v\), and hence \(j'\leq j\). Thus \(u\leq u\vee j'\leq p\). If \(u\vee j'=u\), then \(j'\leq u\leq q\),
contradicting \(q\vee j'=v\). Hence \(u\vee j'=p\), so \(j\leq j'\).
Therefore \(j=j'\). The other equality follows symmetrically.
\end{proof}

For a maximal chain $C$, we define the \emph{join-label set} of $C$
\[
\Lambda(C):=\{\lambda(x,y):x\lessdotrel y\text{ is an edge of }C\}
\subseteq\JJ(L).
\]

\begin{lemma}\label{lem:no-repeat}
The following statements hold:
\begin{enumerate}[label=(\arabic*),font=\normalfont,leftmargin=*]
\item[(1)] Let $x_0\lessdotrel x_1\lessdotrel\cdots\lessdotrel x_r$ be a chain in $L$.  Then the labels $\lambda(x_{i-1},x_i)$ ($i=1,\ldots,r$) are pairwise distinct.
\item[(2)] The set $\Lambda(C)$ is constant on each connected component of $G(L)$.
\end{enumerate}
\end{lemma}
\begin{proof}
(1) Suppose that $j=\lambda(x_{i-1},x_i)$.
Then $j\leq x_i$. Hence, for every $k>i$, we have $j\leq x_{k-1}$.
On the other hand, if $j=\lambda(x_{k-1},x_k)$, then, by the definition of the labeling, $j\nleq x_{k-1}$, a contradiction.

\medskip

(2) We first give the following elementary observation; if maximal chains $C,C'$ are adjacent in $G(L)$, with $C\setminus C'=\{p\}$ and $C'\setminus C=\{q\}$, then there are $u,v\in C\cap C'$ such that \(u\lessdotrel p,q\lessdotrel v\).
Indeed, let $u$ and $v$ be the elements immediately below and above $p$ in $C$.
If $q$ were not strictly between $u$ and $v$, then $p$ could be inserted into $C'$, contradicting maximality.
Therefore, it follows from Lemma~\ref{lem:diamond-label} that $\Lambda(C)=\Lambda(C')$.
\end{proof}

\begin{theorem}\label{thm:char_MD}
Let $L$ be a JSD (resp. MSD) lattice.
Then $L$ is MD (resp. JD) if and only if $G(L)$ is connected.
\end{theorem}

\begin{proof}
Assume that $L$ is JSD.

If $G(L)$ is connected, then all maximal chains have a common join-label set and have the same length from Lemma~\ref{lem:no-repeat}.
Every $j\in\JJ(L)$ labels the edge $j_*\lessdotrel j$.
Extending this edge to a maximal chain $C$ shows $j\in\Lambda(C)$.
Conversely, every label is join-irreducible.
Therefore, we have $\JJ(L)=\Lambda(C)$ for every maximal chain $C$.
From Proposition~\ref{prop:cz}, $L$ is MD.

Conversely, suppose that $L$ is MD.
Then $L$ is JSD and $L^d$ is upper semimodular by Proposition~\ref{prop:cz}.
This shows that $\Delta(L)$ is shellable (\cite[Theorem~6.1]{B80}), and hence $G(L)$ is connected.

The statement in the MSD case follows by duality.
\end{proof}

\begin{corollary}\label{cor:semidist-chain}
Let $L$ be an SD lattice.
Then \(G(L)\) is connected if and only if \(L\) is distributive.
\end{corollary}

\begin{proof}
We can see that $L$ is $MD$ and $JD$ if and only if $L$ is distributive.
Thus, our assertion holds from Theorem~\ref{thm:char_MD}.
\end{proof}

\section{Proof of the main theorem}\label{sec:proof}
We now prove Theorem~\ref{thm:main1}.
By duality, it suffices to consider the join-semidistributive case, which we
assume throughout this section.
Let $L$ be a JSD lattice.

\medskip

%\begin{proof}[Proof of Theorem~\ref{thm:main1}]
We first show Theorem~\ref{thm:main1} (1).
Suppose that $L$ is MD.
As in the proof of Theorem~\ref{thm:char_MD}, $\Delta(L)$ is shellable.
Thus, $L$ is Cohen--Macaulay over an arbitrary field (\cite[Theorem~III.2.5]{StanleyCCA}).

Conversely, if $L$ is Cohen--Macaulay over an arbitrary field, then $G(L)$ is connected by Theorem~\ref{thm:universal}.
Thus, $L$ is MD from Theorem~\ref{thm:char_MD}.

% {\bf (i) $\Rightarrow$ (ii)}: As in the proof of Theorem~\ref{thm:char_MD}, $\Delta(L)$ is shellable discussed since $L$ is MD.
% Thus, $\kk[\Delta(L)]$ is Cohen--Macaulay over any field $\kk$.
% From Theorem~\ref{thm:universal-asl}, we get the assertion.

% {\bf (ii) $\Rightarrow$ (iii)}: Trivial.

% {\bf (iii) $\Rightarrow$ (i)}: Let $\kk$ be a field. Then $\kk[\Delta(L)]$ is an ASL on $L$, and hence satisfies $(S_2)$.
% 

\medskip

Next, we show Theorem~\ref{thm:main1} (2), but it immediately holds from Theorems~\ref{thm:universal} and \ref{thm:char_MD}.

\medskip

To prove parts (3) and (4) of Theorem~\ref{thm:main1}, we construct explicit examples satisfying the desired properties. 
Let $L_{13}$ and $L_7$ be lattices depicted in Figure~\ref{fig:small-lattices}.
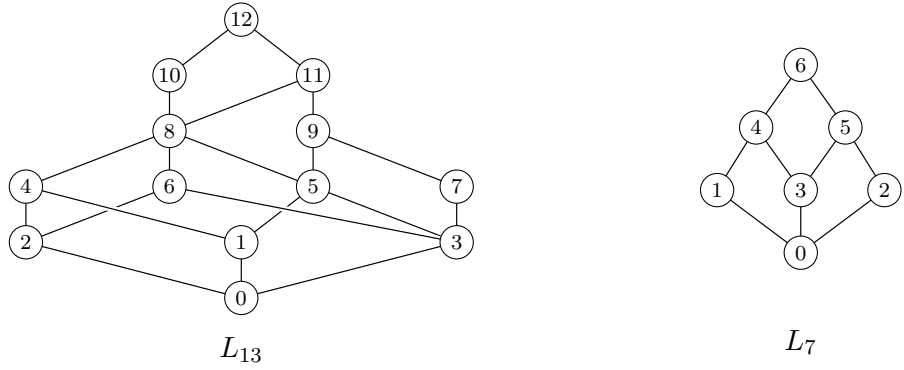
\begin{figure}[H]
\centering
\begin{minipage}[c]{0.62\textwidth}
\centering
\begin{tikzpicture}[
  x=.95cm,y=.64cm,
  cover/.style={line width=.45pt},
  bridge/.style={preaction={draw=white,line width=2.2pt},line width=.45pt},
  vertex/.style={circle,draw,fill=white,minimum size=4.4mm,inner sep=0pt,font=\scriptsize}
]
\coordinate (x0) at (0,0);
\coordinate (x1) at (0,1.15);
\coordinate (x2) at (-3,1.15);
\coordinate (x3) at (3,1.15);
\coordinate (x4) at (-3,2.30);
\coordinate (x5) at (1,2.30);
\coordinate (x6) at (-1,2.30);
\coordinate (x7) at (3,2.30);
\coordinate (x8) at (-1,3.45);
\coordinate (x9) at (1,3.45);
\coordinate (x10) at (-1,4.60);
\coordinate (x11) at (1,4.60);
\coordinate (x12) at (0,5.75);
\draw[cover] (x0)--(x1) (x0)--(x2) (x0)--(x3) (x1)--(x5) (x2)--(x4) (x2)--(x6) (x3)--(x5) (x3)--(x7) (x4)--(x8) (x5)--(x8) (x5)--(x9) (x6)--(x8) (x7)--(x9) (x8)--(x10) (x8)--(x11) (x9)--(x11) (x10)--(x12) (x11)--(x12);
\draw[bridge] (x1)--(x4) (x3)--(x6);
\foreach \i in {0,...,12} \node[vertex] at (x\i) {$\i$};
\end{tikzpicture}

\par
\makebox[\linewidth][c]{%
  \raisebox{-1mm}[0pt][0pt]{\(L_{13}\)}%
}
\end{minipage}
\begin{minipage}[c]{0.35\textwidth}
\centering
\begin{tikzpicture}[
  x=.82cm,y=.72cm,
  cover/.style={line width=.45pt},
  vertex/.style={circle,draw,fill=white,minimum size=4.4mm,inner sep=0pt,font=\scriptsize}
]
\coordinate (y0) at (0,0);
\coordinate (y1) at (-1.35,1.15);
\coordinate (y2) at (1.35,1.15);
\coordinate (y3) at (0,1.15);
\coordinate (y4) at (-.72,2.30);
\coordinate (y5) at (.72,2.30);
\coordinate (y6) at (0,3.45);
\draw[cover] (y0)--(y1) (y0)--(y2) (y0)--(y3) (y1)--(y4) (y2)--(y5) (y3)--(y4) (y3)--(y5) (y4)--(y6) (y5)--(y6);
\foreach \i in {0,...,6} \node[vertex] at (y\i) {$\i$};
\end{tikzpicture}

\par
\makebox[\linewidth][c]{%
  \raisebox{-6mm}[0pt][0pt]{\(L_7\)}%
}
\end{minipage}
\caption{The Hasse diagrams of $L_{13}$ and $L_{7}$.}
\label{fig:small-lattices}
\end{figure}

For any $x\in L_{13}\setminus \{0\}$, we can see that $[x_\downarrow,x]$ is a Boolean lattice, hence $L_{13}$ is MD.
Moreover, using {\tt Macaulay2} (\cite{Macaulay2}), we can compute the $h$-vector
$h(\kk[\Delta(L_{13})])=(1,7,6,1)$.
If $L_{13}$ is integral, there exists a homogeneous ASL domain $R$
on $L_{13}$, which is Cohen--Macaulay by the above discussion.
It follows from \cite[Theorem~2.1]{S91} that the $h$-vector $h(R)=(h_0,\ldots,h_s)$ satisfies 
\begin{equation}\label{eq:stanley-partial-sums}
h_0+\cdots +h_i\leq h_s+\cdots +h_{s-i}\qquad\left(1\leq i\leq\left\lfloor\frac{s}{2}\right\rfloor\right).
\end{equation}
However, from Proposition~\ref{prop:h-vector}, we have $h(R)=h(\kk[\Delta(L_{13})])$ and $8=h_0+h_1\leq h_3+h_2=7$, a contradiction.
This shows Theorem~\ref{thm:main1} (3).

\medskip

It is easy to check that $L_7$ is MD, but not distributive.
Let $\kk$ be a field and let $U=\kk[a,b,c,d]$ be a polynomial ring in 4 variables over $\kk$.
Define the morphism of $\kk$-algebra $\phi : S_{L_7}\to U$ induced by
\[
(x_0,x_1,x_2,x_3,x_4,x_5,x_6)\longmapsto(ab,ab^2,a,abc,ab^2c,ac,d).
\]
{\tt Macaulay2} tells us
\[
I:=\mathrm{Ker}(\phi)=(x_1x_2-x_0^2,\ x_1x_3-x_0x_4,\ x_1x_5-x_0x_3,\ x_2x_3-x_0x_5,\ x_2x_4-x_0x_3,\ x_4x_5-x_3^2).
\]
It is easy to check that the straightening relations given by these six polynomials satisfy condition (\ref{A2}).
Moreover, for the reverse lexicographic order $\prec_{\mathrm{rev}}$ on $S_{L_7}$ induced by $x_0\prec_{\mathrm{rev}} x_1\prec_{\mathrm{rev}} x_2\prec_{\mathrm{rev}} x_3\prec_{\mathrm{rev}} x_4\prec_{\mathrm{rev}} x_5\prec_{\mathrm{rev}} x_6$, these six binomials form a Gr\"obner basis. 
Therefore, we have
\[
\ini_{\prec_{\mathrm{rev}}}(I)=(x_1x_2, x_1x_3, x_1x_5, x_2x_3, x_2x_4, x_4x_5)=J_{L_7},
\]
so (\ref{A1}) holds.
Thus, $S_{L_7}/I\cong \mathrm{Im}(\phi)$ is a homogeneous ASL domain on $L_7$.
Hence $L_7$ is integral, and Theorem~\ref{thm:main1}~(4) follows.

\subsection*{Acknowledgments}
The authors gratefully acknowledge the assistance of generative AI, particularly ChatGPT (OpenAI), during the exploratory and editorial stages of this work. It was used to help organize preliminary drafts, polish the English exposition, and prepare the \LaTeX{} source; all mathematical statements, proofs, and references remain the sole responsibility of the authors.
Koji Matsushita was supported by Grant-in-Aid for JSPS Fellows Grant JP25KJ0047.
Sora Miyashita was supported by Grant-in-Aid for JSPS Fellows Grant JP25KJ1744.    
Koichiro Tani gratefully acknowledges support from JST SPRING, Grant Number JPMJSP2138.

\bibliographystyle{plain}
\bibliography{references}

\end{document}